\documentclass{amsart}

\usepackage{mathrsfs}
\usepackage{graphicx,amssymb,xcolor,enumerate}
\usepackage{hyperref}
\usepackage{verbatim}
\hypersetup{
    bookmarks=true,         % show bookmarks bar?
    pdffitwindow=false,     % window fit to page when opened
    pdfstartview={FitH},    % fits the width of the page to the window
    colorlinks=false,       % false: boxed links; true: colored links
}
 \newtheorem{thm}{Theorem}[section]
 
 \newtheorem{lem}[thm]{Lemma}
 \newtheorem{prop}[thm]{Proposition}
 \theoremstyle{definition}
 \newtheorem{defn}[thm]{Definition}
 \theoremstyle{remark}
 \newtheorem{rem}[thm]{Remark}
 
 \numberwithin{equation}{section}

\usepackage{graphicx}
\usepackage{amssymb, amsmath, amsthm}
\usepackage{amsfonts}
\usepackage{amssymb}
\usepackage{float}
\usepackage{mathrsfs}
\usepackage{enumitem} 
\newcommand{\trho}{\tilde\rho}
\newcommand{\dis}{\displaystyle}

\newcommand{\divv}{\text{\rm div}}

\newcommand{\R}{\mathbb R}

\newcommand{\Int}{\int_0^t\int_{\R^3}}
\newcommand{\intox}{\int_{\R^3}}

\newcommand{\es}{\varepsilon}
\newcommand{\dt}{\partial_t}
\newcommand{\intoxt}{\int_0^T\int_{\R^3}}

\newcommand{\supp}{\text{\rm supp }}

\numberwithin{equation}{section}
\numberwithin{theorem}{section}
\numberwithin{figure}{section}

\begin{document}

%-------------------------------------------------------------------------
% editorial commands: to be inserted by the editorial office
%
%\firstpage{1} \volume{228} \Copyrightyear{2004} \DOI{003-0001}
%\seriesextraline{This is the Concrete Title of this Book\br H.E. R and S.T.C. W, Eds.}
%\firstpage{1}

%\title[Navier-Sokes Equations with Large External Force]
 %{Global Existence of Weak Solution to \\Navier-Stokes Equations with \\Large External Potential Force and General Pressure}

%----------Author 1
%\author[Anthony Suen]{Anthony Suen}

%\address{%
%Department of Mathematics\\Indiana
    %University\\Bloomington, IN 47405}

%\email{cksuen@indiana.edu}

%\thanks{This work was completed with the support of our
%----------Author 2
%\author{A Second Author}
%\address{The address of\br
%the second author\br
%sitting somewhere\br
%in the world}
%\email{dont@know.who.knows}
%\subjclass{35Q30}

%\keywords{Navier-Stokes equations; compressible flow; global weak solutions}

%\date{September 27, 2011}
%\dedicatory{To my family and my wife Candy}
%%% ----------------------------------------------------------------------

\title[Navier-Stokes Equations with Large External Force]
 {Some Serrin type blow-up criteria for the three-dimensional viscous compressible flows with large external potential force}
 
\author{Anthony Suen} 

\address{Department of Mathematics and Information Technology\\
The Education University of Hong Kong}

\email{acksuen@eduhk.hk}

\date{\today}



\begin{abstract}
We provide a Serrin type blow-up criterion for the 3-D viscous compressible flows with large external potential force. For the Cauchy problem of the 3-D compressible Navier-Stokes system with potential force term, it can be proved that the strong solution exists globally if the velocity satisfies the Serrin’s condition and the sup-norm of the density is bounded. Furthermore, in the case of isothermal flows with no vacuum, the Serrin’s condition on the velocity can be removed from the claimed criterion.
\end{abstract}

%%% ----------------------------------------------------------------------
\maketitle
%%% ----------------------------------------------------------------------
%\tableofcontents
\section{Introduction}\label{intro section}

In this present work, we are interested in the 3-D compressible Navier-Stokes equations with an external potential force in the whole space $\R^3$ ($j=1,2,3$):
\begin{align}\label{NS} 
\left\{ \begin{array}{l}
\rho_t + \divv (\rho u) =0, \\
(\rho u^j)_t + \divv (\rho u^j u) + (P)_{x_j} = \mu\,\Delta u^j +
\lambda \, (\divv \,u)_{x_j}  + \rho f^{j}.
\end{array}\right.
\end{align}
Here $x\in\R^3$ is the spatial coordinate and $t\in[0,\infty)$ stands for the time. The unknown functions $\rho=\rho(x,t)$ and $u=(u^1,u^2,u^3)(x,t)$ represent the density and velocity vector in a compressible fluid. The function $P=P(\rho)$ denotes the pressure, $f=(f^1(x),f^2(x),f^3(x))$ is a prescribed external force and $\mu$, $\lambda$ are positive viscosity constants. The system \eqref{NS} is equipped with initial condition
\begin{equation}
\label{IC}
(\rho(\cdot,0)-\trho,u(\cdot,0)) = (\rho_0-\trho,u_0),
\end{equation}
where the non-constant time-independent function $\tilde\rho=\tilde\rho(x)$ (known as the {\it steady state solution} to \eqref{NS}) can be obtained formally by taking $u\equiv0$ in \eqref{NS}:
\begin{align}\label{steady state}
\nabla P(\tilde\rho(x)) =\tilde\rho(x)f(x).
\end{align}

The Navier-Stokes system \eqref{NS} expresses conservation of momentum and conservation of mass for Newtonian fluids, which has been studied by various teams of researchers. The local-in-time existence of classical solution to the full Navier-Stokes equations was proved by Nash \cite{nash} and Tani \cite{tani}. Later, Matsumura and Nishida \cite{mn1} obtained the global-in-time existence of $H^3$-solutions when the initial data was taken to be small (with respect to $H^3$ norm), the results were then generalised by Danchin \cite{danchin} who showed the global existence of solutions in critical spaces. In the case of large initial data, Lions \cite{lions} obtained the existence of global-in-time finite energy weak solutions, yet the problem of uniqueness for those weak solutions remains completely open. In between the two types of solutions as mentioned above, a type of ``intermediate weak'' solutions were first suggested by Hoff in \cite{hoff95, hoff05, hoff06} and later generalised by Matsumura and Yamagata in \cite{MY01}, Suen in \cite{suen13b, suen14, suen16, suen2021existence} and other systems which include compressible magnetohydrodynamics (MHD) \cite{suenhoff12, suen12, suen20b}, compressible Navier-Stokes-Poisson system \cite{suen20a} and chemotaxis systems \cite{LS16}. Solutions as obtained in this intermediate class are less regular than those small-smooth type solutions obtained by Matsumura and Nishida \cite{mn1} and Danchin \cite{danchin} in such a way that the density and velocity gradient may be discontinuous across some hypersurfaces in $\R^3$. On the other hand, those intermediate weak solutions would have more regularity than the large-weak type solutions developed by Lions \cite{lions} so that the uniqueness and continuous dependence of solutions may be obtained; see \cite{hoff06} and other compressible system \cite{suen20b}. 

Nevertheless, the global existence of smooth solution to the Navier-Stokes system \eqref{NS} with arbitrary smooth data is still unknown. From the seminal work given by Xin \cite{xin}, it was proved that smooth solution to \eqref{NS} will blow up in finite time in the whole space when the initial density has compact support. Motivated by the well-known Serrin's criterion on the Leray-Hopf weak solutions to the 3-D incompressible Navier-Stokes equations, Huang, Li and Xin \cite{HLX11} later proved that the strong solution exists globally if the velocity satisfies the Serrin's condition and either the sup-norm of the density or the time integral of the $L^\infty$-norm of the divergence of the velocity is bounded. Under an extra assumption on $\lambda$ and $\mu$, Sun, Wang and Zhang \cite{swz11} obtained a Beale-Kato-Majda blow-up criterion in terms of the upper bound of the density, which is analogous to the Beal-Kato-Majda criterion \cite{BKM84} for the ideal incompressible flows. The results from \cite{swz11} were later generalised to other compressible systems \cite{suen13a, suen15, suen20b}.

In this present work, we extend the results from \cite{HLX11} and \cite{swz11} to the case of compressible Navier-Stokes equations with large potential force. The main novelties of this current work can be summarised as follows:
\begin{itemize}
\item We successfully extend the results from \cite{HLX11} and obtain a Serrin type blow-up criterion for \eqref{NS} in which initial vacuum state is allowed;
\item For the isothermal case, under the assumption that initial density is away from zero, we obtain a blow-up criterion in terms of density {\it only}. Such result is also consistent with the case studied in \cite{suen13a} when the magnetic field is removed. 
\item We introduce some new methods in controlling extra terms originated from the external force $f$ which is absent in \cite{suen13a} and \cite{suen20b}.
\end{itemize} 

We give a brief description on the analysis applied in this work, and the main idea of the following discussion comes from Hoff \cite{hoff95, hoff05, hoff06}. Due to the presence of the external force $f$, one cannot simply apply the same method given in \cite{suen13a} and \cite{suen20b} for obtaining the required blow-up criteria for the solutions. To understand the issue, we consider a decomposition on the velocity $u$ given by 
\begin{equation*}
u=u_p+u_s,
\end{equation*}
for which $u_p$ and $u_s$ satisfy
\begin{align}\label{def of u_p and u_s}  
\left\{ \begin{array}{l}
\mu\Delta(u_p)+\lambda\nabla\divv(u_p)=\nabla(P-P(\trho)),\\
\rho(u_s)_t-\mu\Delta u_s-\lambda\nabla\divv(u_s)=-\rho u\cdot\nabla u-\rho(u_p)_t+\trho^{-1}(\rho-\trho)\nabla P(\trho).
\end{array}\right.
\end{align}
The decomposition of $u$ is important for obtaining some better estimates on the velocity $u$, which allows us to control $u$ in terms of $u_s$ and $u_p$ separately. Since we are addressing solutions around the steady state $(\trho,0)$, it is natural to consider the difference $P-P(\trho)$ as appeared in \eqref{def of u_p and u_s}$_1$. Yet the term $P(\trho)$ will create extra terms in the following sense:
\begin{itemize}
\item On the one hand, since $P(\trho)$ is not necessary a constant, there is an extra term $\nabla P(\trho)$ arising from $\nabla(P-P(\trho))$;
\item On the other hand, using the identity \eqref{steady state}, we can express $f$ in terms of $\trho$ and $P$, so that the term $\rho f$ from \eqref{NS}$_2$ can be combined with $\nabla P(\trho)$ to give $\trho^{-1}(\rho-\trho)\nabla P(\trho)$ in \eqref{def of u_p and u_s}$_2$. 
\end{itemize}
Compare with the previous work \cite{suen13a} and \cite{suen20b}, the term $\trho^{-1}(\rho-\trho)\nabla P(\trho)$ is distinctive for the present system \eqref{NS}, and we have to develop new method for dealing with it. By examining the regularity, one can see that $\trho^{-1}(\rho-\trho)\nabla P(\trho)$ is more regular than $\nabla(P-P(\trho))$, hence it can be used for obtaining $H^1$ estimates on $u_s$ provided that $\|\trho^{-1}(\rho-\trho)\nabla P(\trho)\|_{L^2}$ is under control. Thanks to the $L^2$-energy balance law given by \eqref{L^2 estimate blow up}, we can control $\|\trho^{-1}(\rho-\trho)\nabla P(\trho)\|_{L^2}$ if $\rho$ is bounded. This is a crucial step for obtaining estimates for $u$ in some higher regularity classes, and the details will be carried out in section~\ref{proof of main 2 section}.

Another key of the proof is to extract some ``hidden regularity'' from the velocity $u$ and density $\rho$, which is crucial for decoupling $u$ and $\rho$. In order to achieve our goal, we introduce an important canonical variable $F$ associated with the system \eqref{NS}, which is known as the {\it effective viscous flux}. To see how it works, by the Helmholtz decomposition of the mechanical forces, we can rewrite the momentum equation \eqref{NS}$_2$ as follows (summation over $k$ is understood):
\begin{equation}\label{derivation for F}
\rho\dot u^j = (\trho F)_{x_j}+\mu\omega^{j,k}_{x_k}+\rho f^{j}-P(\trho)_{x_j},
\end{equation}
where $\dot u^j=u^j_t+u\cdot u^j$ is the material derivative on $u^j$, $\omega^{j,k}=u^j_{x_k}-u^k_{x_j}$ is the vorticity and the effective viscous flux $F$ is defined by
\begin{equation*}
\tilde\rho F=(\mu+\lambda)\divv\,u-(P(\rho)-P(\tilde\rho)).
\end{equation*}
By differentiating \eqref{derivation for F} with respect to $x_{j}$ and using the anti-symmetry from $\omega$, we obtain the following Poisson equation for $F$
\begin{equation}\label{poisson in F}
\Delta (\trho F) = \divv(\rho\dot{u}-\rho f+\nabla P(\trho)).
\end{equation}
The Poisson equation \eqref{poisson in F} can be viewed as the analog for compressible Navier-Stokes system of the well-known elliptic equation for pressure in incompressible flow. For sufficiently regular steady state $\trho$, by exploiting the Rankine-Hugoniot condition (see \cite{suenhoff12} for example), one can deduce that the effective viscous flux $F$ is relatively more regular than $\divv(u)$ or $P(\rho)$, which turns out to be crucial for the overall analysis in the following ways:

\noindent{(i)} The equation \eqref{derivation for F} allows us to decompose the acceleration density $\rho \dot u$ as the sum of the gradient of the scalar $F$ and the divergence-free vector field $\omega^{\cdot,k}_{x_k}$. The skew-symmetry of $\omega$ insures that these two vector fields are orthogonal in $L^2$, so that $L^2$ bounds for the terms on the left side of \eqref{derivation for F} immediately give $L^2$ bounds for the gradients of both $F$ and $\omega$. These in turn will be used for controlling $\nabla u$ in $L^4$ when the estimates of $u(\cdot,t)$ in $H^2$ are unknown, which are crucial for estimating different functionals in $u$ and $\rho$; also refer to Lemma~\ref{higher estimate on nabla u lem} and Remark~\ref{rem:L4 bound on u by F}. The details will be carried out in section~\ref{proof of main 1 section}.

\noindent{(ii)} As we have seen before, we aim at applying a decomposition of $u$ given by $u=u_p+u_s$ with $u_p$ satisfying \eqref{def of u_p and u_s}$_1$. To estimate the term $\dt u_p$, if we apply time-derivative on the above identity, then there will be the term $\nabla(\dt P(\rho))$ appearing in the analysis. In view of the strongly elliptic system \eqref{def of u_p and u_s}$_1$, we can obtain estimates on $\|\dt u_p\|_{L^2}$ in terms of the lower order term $\|P(\rho)u\|_{L^2}$ if we have
\begin{align*}
\nabla(\dt P(\rho))=\nabla\divv(-P(\rho)u),
\end{align*}
which is valid when the system is {\it isothermal}, i.e. for the case when $\gamma=1$ in \eqref{condition on P}; also refer to Lemma~\ref{estimate on u_s lemma} and the estimate \eqref{estimate on dt u_p}.
 
We now give a precise formulation of our results. For $r\in[1,\infty]$ and $k\in[1,\infty)$, we let $L^r:=L^r(\R^3)$, $W^{k,r}:=W^{k,r}(\R^3)$ and $H^{k}:=H^{k}(\R^3)$ be the standard Sobolev spaces, and we define the following function spaces for later use (also refer to \cite{HLX11, WEN2013534, swz11} for similar definitions):
\begin{align*} 
\left\{ \begin{array}{l}
D^{k,r}:=\{u\in L^1_{loc}(\R^3):\|\nabla^k u\|_{L^r}<\infty\},\|u\|_{D^{k,r}}:=\|\nabla^k u\|_{L^r}\\
D^{k}:=D^{k,2},D^1_{0}:=\{u\in L^6:\|\nabla u\|_{L^2}<\infty\}.
\end{array}\right.
\end{align*}

We define the system parameters $P$, $f$, $\mu$, $\lambda$ as follows. For the pressure function $P=P(\rho)$ and the external force $f$, we assume that
\begin{align}\label{condition on P}
\mbox{$P(\rho)=a\rho^\gamma$ with constants $a>0$ and $\gamma\ge1$;}
\end{align}
\begin{align}\label{condition on f}
\mbox{there exists $\psi\in H^2$ such that $f=\nabla\psi$ and $\psi(x)\to 0$ as $|x|\to\infty$.}
\end{align}
The viscosity coefficients $\mu$ and $\lambda$ are assumed to satisfy
\begin{align}\label{assumption on viscosity}
7\mu>\lambda>0.
\end{align}

Next, we define $\tilde\rho$ as mentioned at the beginning of this section. Given a constant densty $\rho_{\infty}>0$, we say that $(\tilde\rho,0)$ is a {\it steady state solution} to \eqref{NS} if $\tilde\rho\in C^2(\R^3)$ and the following holds
\begin{align}
\label{eqn for steady state} \left\{ \begin{array}{l}
\nabla P(\tilde\rho(x)) =\tilde\rho(x)\nabla\psi(x), \\
\lim\limits_{|x|\rightarrow\infty}\tilde\rho(x) = \rho_{\infty}.
\end{array}\right.
\end{align}
Given $\rho_\infty>0$, we further assume that
\begin{align}\label{bound on psi}
-\int_0^{\rho_\infty}\frac{P'(\rho)}{\rho} d\rho<\inf_{x\in\R^3} \psi(x)\le\sup_{x\in\R^3} \psi(x)<\int_{\rho_\infty}^\infty \frac{P'(\rho)}{\rho} d\rho,
\end{align} 
and by solving \eqref{eqn for steady state}, $\trho$ can be expressed explicitly as follows:
\[
\trho(x)=
\begin{cases}
\dis\rho_{\infty}\exp(\frac{1}{a}\psi(x)), & \text{for } \gamma=1 \\
    \dis(\rho_{\infty}^{\gamma-1}+\frac{\gamma-1}{a\gamma}\psi(x))^\frac{1}{\gamma-1}, & \text{for } \gamma>1.
\end{cases}
\]
We recall a useful lemma from \cite{LM11} about the existence of steady state solution $(\tilde\rho,0)$ to  \eqref{NS} which can be stated as follows:
\begin{lem}\label{existence for steady state lem}
Given $\rho_\infty>0$, if we assume that $P$, $f$, $\psi$ satisfy \eqref{condition on P}-\eqref{assumption on viscosity} and \eqref{bound on psi}, then there exists positive constants $\rho_1,\rho_2,\delta$ and a unique solution $\tilde\rho$ of \eqref{eqn for steady state} satisfying $\tilde\rho-\rho_\infty\in H^2\cap W^{2,6}$ and 
\begin{align}\label{1.1.17-2}
\rho_1<\rho_1+\delta\le\tilde\rho\le\rho_2-\delta<\rho_2.
\end{align}
\end{lem}
From now on, we fix $\rho_\infty>0$ and choose $\trho$ satisfying Lemma~\ref{existence for steady state lem}. And for the sake of simplicity, we also write $P=P(\rho)$ and $\tilde P=P(\trho)$ unless otherwise specified.

We give the definitions for strong solution and maximal time of existence as follows.

\begin{defn}
We say that $(\rho,u)$ is a (local) {\it strong solution} of \eqref{NS} if for some $T>0$ and $q'\in(3,6]$, we have
\begin{align}\label{def of strong sol} 
\left\{ \begin{array}{l}
0\le\rho\in C([0,T],W^{1,q'}),\qquad\rho_t\in C([0,T],L^{q'}),\\
u\in C([0,T], D^1\cap D^2)\cap L^2(0,T; D^{2,q'}),\\
\rho^\frac{1}{2}u_t\in L^\infty(0,T;L^2),\qquad u_t\in L^2(0,T;D^1).
\end{array}\right.
\end{align}
Furthermore, $(\rho,u)$ satisfy the following conditions: 
\begin{itemize}
\item For all $0\le t_1\le t_2\le T$ and $C^1$ test functions $\varphi\in\mathcal{D}(\R^3\times(-\infty,\infty))$ which are Lipschitz on $\R^3\times[t_1,t_2]$ with $\supp\varphi(\cdot,\tau)\subset K$, $\tau\in[t_1,t_2]$, where $K$ is compact, it holds 
\begin{align}\label{WF1}
\left.\int_{\R^3}\rho(x,\cdot)\varphi(x,\cdot)dx\right|_{t_1}^{t_2}=\int_{t_1}^{t_2}\int_{\R^3}(\rho\varphi_t + \rho u\cdot\nabla\varphi)dxd\tau.
\end{align}
\item For test functions $\varphi$ which are locally Lipschitz on $\R^3 \times [0, T]$ and for which $\varphi$, $\varphi_t$, $\nabla\varphi \in L^2(\R^3 \times (0,T))$, $\nabla\varphi \in L^\infty(\R^3 \times (0,T))$ and $\varphi(\cdot,T) = 0$, it holds
\begin{align}\label{WF2}
\left.\int_{\R^3}(\rho u^{j})(x,\cdot)\varphi(x,\cdot)dx\right|_{t_1}^{t_2}=&\int_{t_1}^{t_2}\int_{\R^3}[\rho u^{j}\varphi_t + \rho u^{j}u\cdot\nabla\varphi + (P-\tilde{P})\varphi_{x_j}]dxd\tau\notag\\
& - \int_{t_1}^{t_2}\int_{\R^3}[\mu\nabla u^{j}\cdot\nabla\varphi + \lambda(\divv(u))\varphi_{x_j}]dxd\tau\\
&+ \int_{t_1}^{t_2}\int_{\R^3}(\rho f - \nabla \tilde{P})\cdot\varphi dxd\tau.\notag
\end{align}
\end{itemize}
\end{defn}

\begin{defn}
We define $T^*\in(0, \infty)$ to be the {\it maximal time of existence} of a strong solution $(\rho,u)$ to \eqref{NS} if for any $0<T<T^*$, $(\rho,u)$ solves \eqref{NS} in $[0, T]\times\R^3$ and satisfies \eqref{def of strong sol}-\eqref{WF2}. Moreover, the conditions \eqref{def of strong sol}-\eqref{WF2} fail to hold when $T=T^*$.
\end{defn}

We are ready to state the following main results of this paper which are summarised in Theorem~\ref{Main thm for gamma>1}-\ref{Main thm for gamma=1}:

\begin{thm}\label{Main thm for gamma>1}
Given $\rho_\infty>0$, let $\trho$ be the steady state solution to \eqref{eqn for steady state}. Let $(\rho,u)$ be a strong solution to the Cauchy problem \eqref{NS} satisfying \eqref{condition on P}-\eqref{assumption on viscosity} with $\gamma>1$. Assume that the initial data $(\rho_0,u_0)$ satisfy
\begin{align}\label{condition on IC1}
\rho_0\ge0,\qquad\rho_0-\trho\in L^1\cap H^1\cap W^{1,\tilde q},\qquad u_0\in D^1\cap D^2,
\end{align}
for some $\tilde q>3$ and the compatibility condition
\begin{align}\label{compatibility condition}
-\mu\Delta u_0-\lambda\nabla\divv(u_0)+\nabla P(\rho_0)-\rho_0 f=\rho_0^\frac{1}{2}g,
\end{align}
for some $g \in L^2$. If $T^*<\infty$ is the maximal time of existence, then we have
\begin{align}\label{blow-up 1}
\lim_{T\to T^*}(\sup_{\R^3\times[0,T]}|\rho|+\|\rho^\frac{1}{2}u\|_{L^s(0,T;L^r)})=\infty,
\end{align}
for some $r$, $s$ that satisfy
\begin{align}\label{conditions on r s}
\frac{2}{s}+\frac{3}{r}\le1,\qquad r\in(3,\infty],\qquad s>\frac{3}{2}.
\end{align}
\end{thm}

\begin{thm}\label{Main thm for gamma=1}
Let $(\rho,u)$ be a strong solution to the Cauchy problem \eqref{NS} satisfying \eqref{condition on P}-\eqref{assumption on viscosity} with $\gamma=1$. Assume that the initial data $(\rho_0,u_0)$ satisfy \eqref{condition on IC1}-\eqref{compatibility condition}. Suppose that the initial density $\rho_0$ further satisfies
\begin{equation}\label{further assumption on initial density}
\inf_{x\in\R^3}\rho_0(x)>0.
\end{equation}
If $T^*<\infty$ is the maximal time of existence, then we have
\begin{align}\label{blow-up 2}
\lim_{T\to T^*}\sup_{\R^3\times[0,T]}|\rho|=\infty.
\end{align}
\end{thm}

The rest of the paper is organised as follows. In section~\ref{prelim section}, we recall some known facts and useful inequalities which will be used in later analysis. In section~\ref{proof of main 1 section}, we give the proof of Theorem~\ref{Main thm for gamma>1} by obtaining some necessary bounds on the strong solutions. In section~\ref{proof of main 2 section}, we give the proof of Theorem~\ref{Main thm for gamma=1} by introducing a different approach for the isothermal case $\gamma=1$.

\section{Preliminaries}\label{prelim section}

In this section, we give some known facts and useful inequalities. We first state the following local-in-time existence and uniqueness of strong solutions to \eqref{NS} with non-negative initial density (references can be found in \cite{nash} and \cite{tani}):

\begin{prop}\label{Local-in-time existence prop}
Let $\trho$ be a steady state solution and $(\rho_0-\trho,u_0)$ be given initial data satisfying \eqref{condition on IC1}-\eqref{compatibility condition}, then there exists a positive time $T>0$ and a unique strong solution $(\rho,u)$ to \eqref{NS} defined on $\R^3\times(0,T]$.
\end{prop}

Next, we recall the following Gagliardo-Nirenberg inequalities:

\begin{prop}
For $p\in[2,6]$, $q\in(1,\infty)$ and $r\in(3,\infty)$, there exists some generic constant $C>0$ such that for any $h_1\in H^1$ and $h_2\in L^q\cap D^{1,r}$, we have
\begin{align}
\|h_1\|^p_{L^p}&\le C\|h_1\|^\frac{6-p}{2}_{L^2}\|\nabla h_1\|^\frac{3p-6}{2}_{L^2},\label{GN1}\\
\|h_2\|_{L^\infty}&\le C\|h_2\|^\frac{q(r-3)}{3r+q(r-3)}_{L^q}\|\nabla h_2\|^\frac{3r}{3r+q(r-3)}_{L^r}.\label{GN2}
\end{align}
\end{prop}

We also recall the following two canonical functions, namely the effective viscous flux $F$ and vorticity $\omega$, which are defined by
\begin{equation}\label{def of F and omega}
\tilde\rho F=(\mu+\lambda)\divv\,u-(P(\rho)-P(\tilde\rho)),\qquad\omega=\omega^{j,k}=u^j_{x_k}-u^k_{x_j}.
\end{equation}
The following lemma gives some useful estimates on $u$ in terms of $F$ and $\omega$.
\begin{lem}
For $r_1,r_2\in(1,\infty)$ and $t>0$, there exists a universal constant $C$ which depends on $r_1$, $r_2$, $\mu$, $\lambda$, $a$, $\gamma$ and $\trho$ such that, the following estimates hold:
\begin{align}\label{bound on F and omega in terms of u}
\|\nabla F\|_{L^{r_1}}+\|\nabla\omega\|_{L^{r_1}}&\le C(\|\rho^\frac{1}{2} \dot{u}\|_{L^{r_1}}+\|(\rho-\trho)\|_{L^{r_1}})\notag\\
&\le C(\|\rho^\frac{1}{2} u_t\|_{L^{r_1}}+\|\rho^\frac{1}{2} u\cdot\nabla u\|_{L^{r_1}}+\|(\rho-\trho)\|_{L^{r_1}})
\end{align}
\begin{align}\label{bound on u in terms of F and omega}
||\nabla u(\cdot,t)||_{L^{r_2}}\le C(|| F(\cdot,t)||_{L^{r_2}}+||\omega(\cdot,t)||_{L^{r_2}}+||(\rho-\tilde\rho)(\cdot,t)||_{L^{r_2}}).
\end{align}
\end{lem}
\begin{proof}
By the definitions of $F$ and $\omega$, and together with \eqref{NS}$_2$, $F$ and $\omega$ satisfy the elliptic equations
\begin{align*}
\Delta (\tilde\rho F)=\divv(\rho \dot{u}-\rho f+\nabla P(\tilde\rho))=\divv(\rho u_t+\rho u\cdot\nabla u-\rho f+\nabla P(\tilde\rho)),
\end{align*}
\begin{align*}
\mu\Delta\omega=\nabla\times(\rho \dot{u}-\rho f+\nabla P(\tilde\rho))=\nabla\times(\rho u_t+\rho u\cdot\nabla u-\rho f+\nabla P(\tilde\rho)),
\end{align*}
where $\dot{h}:=\dt h+u\cdot \nabla h$ is the material derivative on $h$. Hence by applying standard $L^p$-estimate, the estimates \eqref{bound on F and omega in terms of u}-\eqref{bound on u in terms of F and omega} follow.
\end{proof}
Finally, we recall the following inequality which was first proved in \cite{BKM84} for the case $\divv(u)\equiv0$ and was proved in \cite{HLX11} for compressible flows.
\begin{prop}
For $q\in(3,\infty)$, there is a positive constant $C$ which depends on $q$ such that the following estimate holds for all $\nabla u\in L^2\cap D^{1,q}$,
\begin{align}\label{log estimate on nabla u}
\|\nabla u\|_{L^\infty}\le C(&\|\divv(u)\|_{L^\infty}+\|\nabla\times u\|_{L^\infty}\ln(e+\|\nabla^2 u\|_{L^q})\notag\\
&+C\|\nabla u\|_{L^2}+C,
\end{align}
where $e$ is the base of the natural logarithm.
\end{prop}

\section{Proof of Theorem~\ref{Main thm for gamma>1}}\label{proof of main 1 section}

In this section, we give the proof of Theorem~\ref{Main thm for gamma>1}. Let $(\rho,u)$ be a strong solution to the system \eqref{NS} as described in Theorem~\ref{Main thm for gamma>1}. By performing standard $L^2$-energy estimate (see \cite{suen2021existence} for example), we readily have
\begin{align}\label{L^2 estimate blow up}
\sup_{0\le \tau\le t}\left((\|\rho^\frac{1}{2} u(\cdot,\tau)\|_{L^2}^2+\intox G(\rho(x,s))dx\right) + \int_{0}^{t}\|\nabla u(\cdot,\tau)\|_{L^2}^2d\tau\le C_0,
\end{align}
for all $t\in[0,T^*)$, where $C_0$ depends on the initial data but is independent of both $t$ and $T^*$. Here $G$ is a functional given by
\begin{align*}
\rho\int_{\trho}^\rho\frac{P(s)-P(\trho)}{s^2}ds=\rho\int_{\trho}^\rho\frac{as^\gamma-a\trho^\gamma}{s^2}ds.
\end{align*}
In order to prove Theorem~\ref{Main thm for gamma>1}, for the sake of contradiction, suppose that \eqref{blow-up 1} does not hold. Then there exists some constant $M_0>0$ such that
\begin{align}\label{blow-up 1 not}
\lim_{T\to T^*}(\sup_{\R^3\times[0,T]}|\rho|+\|\rho^\frac{1}{2}u\|_{L^s(0,T;L^r)})\le M_0.
\end{align}
We first obtain the estimates on $\nabla u$ and $u_t$ under \eqref{blow-up 1 not}:
\begin{lem}\label{estimate on nabla u lem}
Assume that \eqref{blow-up 1 not} holds, then for $t\in[0,T^*)$, we have
\begin{align}\label{estimate on nabla u}
\sup_{0\le \tau\le t}\|\nabla u(\cdot,\tau)\|^2_{L^2}+\Int\rho|u_t|^2dxd\tau\le C,
\end{align}
where and in what follows, $C$ denotes a generic constant which depends on $\mu$, $\lambda$, $a$, $f$, $\gamma$, $\trho$, $M_0$, $T^*$ and the initial data.
\end{lem}

\begin{proof}
We multiply the momentum equation \eqref{NS}$_2$ by $u_t$ and integrate to obtain
\begin{align*}
&\frac{1}{2}\frac{d}{dt}\intox(\mu|\nabla u|^2+\lambda(\divv(u))^2)dx+\intox\rho|u_t|^2dx\notag\\
&=\intox P\divv(u)dx-\intox\rho u\cdot\nabla u\cdot u_tdx+\intox\rho f\cdot u_tdx.
\end{align*}
Using H\"{o}lder's inequality and Young's inequality, the term involving $f$ can be bounded by
\begin{align*}
\Big|\intox\rho f\cdot u_tdx\Big|&\le\Big(\intox\rho|u_t|^2dx\Big)^\frac{1}{2}\Big(\intox\rho|f|^2dx\Big)^\frac{1}{2}\\
&\le\frac{1}{2}\Big(\intox\rho|u_t|^2dx\Big)+C\Big(\intox\rho|f|^2dx\Big).
\end{align*}
Hence by following the steps given in \cite{HLX11}, we arrive at
\begin{align}\label{bound on nabla u 1}
\frac{1}{2}\frac{d}{dt}\intox(|\nabla u|^2+(\divv(u))^2 - P\divv(u))dx+\frac{1}{2}\intox\rho|u_t|^2dx\notag\\
\le C\|\nabla u\|^2_{L^2}+C\intox\rho|u\cdot\nabla u|^2dx+C.
\end{align}
To estimate the advection term on the right side of \eqref{bound on nabla u 1}, for $r$, $s$ satisfying \eqref{conditions on r s}, we use \eqref{GN1} and \eqref{bound on u in terms of F and omega} to obtain
\begin{align*}
\|\rho^\frac{1}{2} u\cdot\nabla u\|_{L^2}&\le C\|\rho u\|_{L^r}\|\nabla u\|_{L^\frac{2r}{r-2}}\\
&\le C\|\rho^\frac{1}{2}u\|_{L^r}(\|F\|^{1-\frac{3}{r}}_{L^2}\|\nabla F\|^\frac{3}{r}_{L^2}+\|\nabla\omega\|^{1-\frac{3}{r}}_{L^2}\|\nabla \omega\|^\frac{3}{r}_{L^2}+1).
\end{align*}
Using Young's inequality, for any $\es>0$ being small, there exists $C_{\es}>0$ such that
\begin{align*}
&\|\rho^\frac{1}{2}u\|_{L^r}(\|F\|^{1-\frac{3}{r}}_{L^2}\|\nabla F\|^\frac{3}{r}_{L^2}+\|\nabla\omega\|^{1-\frac{3}{r}}_{L^2}\|\nabla \omega\|^\frac{3}{r}_{L^2}+1)\\
&\le\varepsilon(\|\nabla F\|_{L^2}+\|\nabla\omega\|_{L^2}))+C_{\es}\|\rho^\frac{1}{2}u\|^\frac{s}{2}_{L^r}(\|F\|_{L^2}+\|\omega\|_{L^2}+1)+C_{\es},
\end{align*}
and hence by applying \eqref{bound on F and omega in terms of u}, we obtain
\begin{align}\label{bound on nabla u 2}
\|\rho^\frac{1}{2} u\cdot\nabla u\|_{L^2}\le C\varepsilon\|\rho u_t\|_{L^2}+C_{\es}\|\rho^\frac{1}{2}u\|^\frac{s}{2}_{L^r}(\|\nabla u\|_{L^2}+1)+C_{\es}.
\end{align}
Applying \eqref{bound on nabla u 2} on \eqref{bound on nabla u 1} and choosing $\es>0$ small enough,
\begin{align}\label{bound on nabla u 3}
&\frac{d}{dt}\intox(|\nabla u|^2+(\divv(u))^2 - P\divv(u))dx+\frac{1}{2}\intox\rho|u_t|^2dx\notag\\
&\le C(\|\rho^\frac{1}{2}u\|^\frac{s}{2}_{L^r}+1)(\|\nabla u\|_{L^2}+1)+C\le C(\|\nabla u\|_{L^2}+1),
\end{align}
where the last inequality follows by \eqref{blow-up 1 not}. Hence the estimate \eqref{estimate on nabla u} follows by using Gr\"{o}nwall's inequality on \eqref{bound on nabla u 3}.
\end{proof}
Next, we make use of Lemma~\ref{estimate on nabla u lem} to obtain some higher order estimates on $u$ which can be stated in the following lemma:
\begin{lem}\label{higher estimate on nabla u lem}
Assume that \eqref{blow-up 1 not} holds, then for all $t\in[0,T^*)$, we have
\begin{align}\label{higher estimate on nabla u}
\sup_{0\le \tau\le t}\rho\|\dot u(\cdot,\tau)\|^2_{L^2}+\Int\rho|\nabla\dot{u}|^2dxd\tau\le C.
\end{align}
\end{lem}
\begin{proof}
Following the steps given in \cite{HLX11}, we readily have
\begin{align*}
\sup_{0\le \tau\le t}\rho\|\dot u(\cdot,\tau)\|^2_{L^2}+\Int\rho|\nabla\dot{u}|^2dxd\tau\le C\int_0^t\|\nabla u(\cdot,\tau)\|^4_{L^4}d\tau+C\Big(\intox\rho|f|^2\Big)+C.
\end{align*}
To estimate the term $\dis\int_0^t\|\nabla u\|^4_{L^4}d\tau$, we apply \eqref{bound on F and omega in terms of u} and \eqref{bound on u in terms of F and omega} to get
\begin{align*}
\int_0^t\|\nabla u\|^4_{L^4}d\tau&\le C\int_0^t(\|F\|^4_{L^4}+\|\omega\|^4_{L^4})d\tau+C\\
&\le C\int_0^t(\|F\|^\frac{5}{2}_{L^2}\|\nabla F\|^\frac{3}{2}_{L^2}+\|\omega\|^\frac{5}{2}_{L^2}\|\nabla \omega\|^\frac{3}{2}_{L^2})d\tau+C\\
&\le C\int_0^t\|\nabla\dot{u}\|^\frac{3}{2}_{L^2}d\tau+C,
\end{align*}
hence by using Young's inequality and Gr\"{o}nwall's inequality, the estimate \eqref{higher estimate on nabla u} holds for all $T\in[0,T^*)$.
\end{proof}

\begin{rem}\label{rem:L4 bound on u by F}
As pointed out in section~\ref{intro section}, it is important to use the effective viscous flux $F$ and the vorticity $\omega$ for estimating the term $\dis\int_0^t\|\nabla u\|^4_{L^4}d\tau$, since there is no available {\it a priori} bound on $\|\nabla u\|_{H^1}$ and hence, we cannot merely apply the Sobolev embedding $H^1\hookrightarrow L^4$ in Lemma~\ref{higher estimate on nabla u lem}.
\end{rem}

We give the following estimate on the density gradient $\nabla \rho$ and the $H^1$ norm of $\nabla u$:
\begin{lem}\label{higher estimate on rho and u lem}
Assume that \eqref{blow-up 1 not} holds, then for all $t\in[0,T^*)$, we have
\begin{align}\label{higher estimate on rho and u}
\sup_{0\le \tau\le t}(\|\rho\|_{H^1\cap W^{1,q'}}+\|\nabla u\|_{H^1})(\cdot,\tau)\le C,
\end{align}
for all $q'\in(3,6]$.
\end{lem}
\begin{proof}
For any $p\in[2,6]$, we have
\begin{align*}
&\frac{d}{dt}(|\nabla\rho|^p)+\divv(|\nabla(\rho-\trho)|^pu)+(p-1)|\nabla(\rho-\trho)|^p\divv(u)\\
&+p|\nabla(\rho-\trho)|^{p-2}\nabla(\rho-\trho)\nabla u\nabla(\rho-\trho)+p(\rho-\trho)|\nabla(\rho-\trho)|^{p-2}\nabla(\rho-\trho)\nabla\divv(u)\\
&=-p\nabla\divv(\trho u)\cdot\nabla(\rho-\trho)|\nabla(\rho-\trho)|^{p-2}.
\end{align*}
We integrate the above equation over $\R^3$ and use \eqref{GN1}, \eqref{bound on F and omega in terms of u} and \eqref{higher estimate on nabla u} to obtain
\begin{align}\label{estimate on dt nabla rho}
&\frac{d}{dt}\|\nabla(\rho-\trho)\|_{L^p}\notag\\
&\le C(1+\|\nabla u\|_{L^\infty}+\|\nabla(\trho F)\|_{L^p}+\|u\|_{L^p}+\|\nabla u\|_{L^p})\|\nabla(\rho-\trho)\|_{L^p}\notag\\
&\le C(1+\|\nabla u\|_{L^\infty}+\|\nabla\dot{u}\|_{L^2})\|\nabla(\rho-\trho)\|_{L^p}.
\end{align}
On the other hand, upon rearranging terms from the momentum equation \eqref{NS}$_2$, we have
\begin{align}\label{elliptic system for u}
\mu\Delta u+\lambda\nabla\divv(u)=\rho\dot{u}+\nabla(P(\rho)-P(\trho))+\nabla P(\trho)(\trho-\rho)\trho^{-1}.
\end{align}
Hence for each $q'\in(3,6]$, by applying $L^{q'}$-estimate on $u$ in \eqref{elliptic system for u}, we have
\begin{align}\label{q' estimate on nabla u}
\|\nabla u\|_{W^{1,q'}}\le C(1+\|\nabla\dot{u}\|_{L^2}+\|\nabla(\rho-\trho)\|_{L^{q'}}).
\end{align}
Using the Sobolev inequality \eqref{GN2}, together with the estimates \eqref{log estimate on nabla u} and \eqref{q' estimate on nabla u}, we have
\begin{align}\label{log estimate on nabla u with rho}
\|\nabla u\|_{L^\infty}\le C+&C(\|\divv(u)\|_{L^\infty}+\|\omega\|_{L^\infty})\ln(e+\|\nabla\dot{u}\|_{L^2})\notag\\
&+C(\|\divv(u)\|_{L^\infty}+\|\omega\|_{L^\infty})\ln(e+\|\nabla(\rho-\trho)\|_{L^{q'}}).
\end{align}
To estimate the time integral of $(\|\divv(u)\|_{L^\infty}^2+\|\omega\|_{L^\infty}^2)$, using \eqref{GN2}, \eqref{bound on F and omega in terms of u} and \eqref{higher estimate on nabla u}, we readily have
\begin{align}\label{bound on time integral on div(u) and omega}
\int_0^t(\|\divv(u)\|_{L^\infty}^2+\|\omega\|_{L^\infty}^2)(\cdot,\tau)d\tau&\le C\int_0^t(\|F\|^2_{L^\infty}+\|\omega\|^2_{L^\infty})(\cdot,\tau)d\tau+C\notag\\
&\le C\int_0^t\|\nabla\dot{u}(\cdot,\tau)\|^2_{L^2}d\tau+C\le C.
\end{align}
Hence by applying \eqref{log estimate on nabla u with rho} on \eqref{estimate on dt nabla rho} with $p=q'$, using Gr\"{o}nwall's inequality with the bounds \eqref{higher estimate on nabla u} and \eqref{bound on time integral on div(u) and omega}, we obtain
\begin{align}\label{sup bound on Lq rho}
\sup_{0\le \tau\le t}\|\nabla(\rho-\trho)(\cdot,\tau)\|_{L^{q'}}\le C.
\end{align}
By combining \eqref{log estimate on nabla u with rho} with \eqref{sup bound on Lq rho} and \eqref{bound on time integral on div(u) and omega}, it further gives
\begin{align}\label{bound on time integral of nabla u infty}
\int_0^t\|\nabla u(\cdot,\tau)\|_{L^\infty}d\tau\le C.
\end{align}
Imtegraing \eqref{estimate on dt nabla rho} with $p=2$ over $t$ and together with \eqref{higher estimate on nabla u} and \eqref{bound on time integral of nabla u infty}, it follows that
\begin{align}\label{sup bound on L2 rho}
\sup_{0\le \tau\le t}\|\nabla(\rho-\trho)(\cdot,\tau)\|_{L^2}\le C.
\end{align}
which gives the bound on $\rho-\trho$ as claimed in \eqref{higher estimate on rho and u}. The bound on $u$ as appeared in \eqref{higher estimate on rho and u} then follows from $L^2$-estimate on \eqref{elliptic system for u} with the bounds \eqref{estimate on nabla u} and \eqref{sup bound on L2 rho}, and we finish the proof for \eqref{higher estimate on rho and u}.
\end{proof}
\begin{proof}[Proof of Theorem~\ref{Main thm for gamma>1}]
The proof then follows from the same argument given in \cite{HLX11}, namely by choosing the function $(\rho,u)(x,T^*)$ to be the limit of $(\rho,u)(x,t)$ as $t\to T^*$, one can show that $(\rho,u)(x,T^*)$ satisfies the compatibility condition \eqref{compatibility condition} as well. Therefore, if we take $(\rho,u)(x,T^*)$ to be the new initial data for the system \eqref{NS}, then Proposition~\ref{Local-in-time existence prop} applies and shows that the local strong solution can be extended beyond the maximal time $T^*$.  
\end{proof}

\section{Proof of Theorem~\ref{Main thm for gamma=1}}\label{proof of main 2 section}

In this section, we prove Theorem~\ref{Main thm for gamma=1} using a different approach compared with the proof of Theorem~\ref{Main thm for gamma>1}. We let $(\rho,u)$ be a strong solution to the system \eqref{NS} for the isothermal case as described in Theorem~\ref{Main thm for gamma=1}, and for the sake of contradiction, suppose that \eqref{blow-up 2} does not hold. Then there exists a constant $\tilde{M}_0>0$ such that
\begin{align}\label{blow-up 2 not}
\lim_{T\to T^*}\sup_{\R^3\times[0,T]}|\rho|\le \tilde{M}_0.
\end{align}
Furthermore, together with the bound \eqref{bound on time integral of nabla u infty} on $\|\nabla u\|_{L^\infty}$ and the assumption \eqref{further assumption on initial density} on $\rho_0$, we have
\begin{align}\label{lower bound on rho}
\inf_{\R^3\times[0,T^*)}\rho\ge \tilde{M}_1,
\end{align}
where $\tilde{M}_1>0$ is a constant which depends on $\mu$, $\lambda$, $a$, $f$, $\trho$, $\tilde{M}_0$, $T^*$ and the initial data.

To facilitate our discussion, we introduce the following auxiliary functionals:
\begin{align}
\Phi_1(t)&=\sup_{0\le \tau\le t}\|\nabla u(\cdot,\tau)\|^2_{L^2}+\int_{0}^{t}\|\rho^\frac{1}{2}\dot{u}(\cdot,\tau)\|_{L^2}^2d\tau,\label{def of Phi 1}\\
\Phi_2(t)&=\sup_{0\le \tau\le t}\|\rho^\frac{1}{2}\dot{u}(\cdot,\tau)\|^2_{L^2}+\int_{0}^{t}\|\nabla\dot{u}(\cdot,\tau)\|^2_{L^2}d\tau,\label{def of Phi 2}\\
\Phi_3(t)&=\int_0^t\int_{\R^3}|\nabla u|^4dxd\tau.\label{def of Phi 3}
\end{align}

We recall the following lemma which gives estimates on the solutions of the Lam\'{e} operator $\mu\Delta+\lambda\nabla\divv$. Details can be found in \cite[pp. 39]{swz11}.
\begin{lem}\label{estimates on Lame operator}
Consider the following equation:
\begin{equation}\label{eqn for Lame operator}
\mu\Delta v+\lambda\nabla\divv(v)=J,
\end{equation}
where $v=(v^1,v^2,v^3)(x)$, $J=(J^1,J^2,J^3)(x)$ with $x\in\R^3$ and $\mu$, $\lambda>0$. Then for $r\in(1,\infty)$, we have:
\begin{itemize} 
\item if $J\in W^{2,r}(\R^3)$, then $||\Delta v||_{L^r}\le C||J||_{L^r}$;
\item if $J=\nabla\varphi$ with $\varphi\in W^{2,r}(\R^3)$, then $||\nabla v||_{L^r}\le C||\varphi||_{L^r}$;
\item if $J=\nabla\divv(\varphi)$ with $\varphi\in W^{2,r}(\R^3)$, then $||v||_{L^r}\le C||\varphi||_{L^r}$.
\end{itemize}
Here $C$ is a positive constant which depends on $\mu$, $\lambda$ and $r$.
\end{lem}
One of the key for the proof of Theorem~\ref{Main thm for gamma=1} is to estimate the $L^4$ norm of $\rho^\frac{1}{4} u$ and the results can be summarised in the following lemma:
\begin{lem}
Assume that \eqref{blow-up 2 not} holds, then for $t\in[0,T^*)$, we have
\begin{align}\label{L4 estimate on u}
\sup_{0\le \tau\le t}\intox\rho|u|^4dx\le \tilde C,
\end{align}
where and in what follows, $\tilde C$ denotes a generic constant which depends on $\mu$, $\lambda$, $a$, $f$, $\trho$, $\tilde{M}_0$, $T^*$, $\tilde{M}_1$ and the initial data.
\end{lem}
\begin{proof}
It can be proved by the method given in \cite{HLX11} (also refer to \cite{hoff95} and \cite{HL09} for more details) and we omit here for the sake of brevity. We point out that the condition \eqref{assumption on viscosity} is required for obtaining \eqref{L4 estimate on u}.
\end{proof}
We begin to estimate the functionals $\Phi_1$, $\Phi_2$ and $\Phi_3$. The following lemma gives an estimate on $\Phi_1$ in terms of $\Phi_3$:
\begin{lem}\label{bound on Phi 1 lemma}
Assume that \eqref{blow-up 2 not} holds. For any $0\le t< T^*$,
\begin{equation}\label{bound on Phi 1}
\Phi_1(t)\le \tilde C[1+\Phi_3(t)].
\end{equation}
\end{lem}
\begin{proof}
Following the proof of Lemma~\ref{estimate on nabla u lem}, we have
\begin{align}\label{H1 estimate blow up}
\sup_{0\le \tau\le t}\intox|\nabla u|^2dxd\tau+\intoxt\rho|\dot{u}|^2dxd\tau\le \tilde C+\tilde C\intoxt|\nabla u|^3dxd\tau.
\end{align}
The second term on the right side of \eqref{H1 estimate blow up} can be bounded by
\[\begin{aligned}
\intoxt|\nabla u|^3dxd\tau&\le \Big(\intoxt|\nabla u|^2dxd\tau\Big)^\frac{1}{2}\Big(\intoxt|\nabla u|^4dxd\tau\Big)^\frac{1}{2}\\
&\le \tilde C\Phi_3^\frac{1}{2}.
\end{aligned}\]
Applying the above bounds on \eqref{H1 estimate blow up}, the result follows.
\end{proof}
Before we estimate $\Phi_2$, we introduce the following decomposition on $u$ stated in section~\ref{intro section}. We write
\begin{equation}\label{decomposition on u}
u=u_p+u_s,
\end{equation}
where $u_p$ and $u_s$ satisfy \eqref{def of u_p and u_s} and we recall that $\tilde P:=P(\trho)$. Then by using \eqref{estimates on Lame operator}, for all $r>1$, the term $u_p$ can be bounded by
\begin{equation}\label{estimate on u_p}
\intox|\nabla u_p|^rdx\le \tilde C\intox|P-\tilde P|^rdx\le \tilde C\intox|\rho-\tilde\rho|^rdx.
\end{equation}
On the other hand, the term $u_s$ can be estimated as follows.
\begin{lem}\label{estimate on u_s lemma}
For any $0\le t<T^*$, we have
\begin{equation}\label{estimate on u_s}
\sup_{0\le \tau\le t}\intox|\nabla u_s|^2dxd\tau+\intoxt\rho|\dt(u_s)|^2dxd\tau+\intoxt|\Delta u_s|^2dxd\tau\le \tilde C.
\end{equation}
\end{lem}
\begin{proof}
We multiply \eqref{def of u_p and u_s}$_2$ by $\dt(u_s)$ and integrate to obtain
\begin{align}\label{estimate on u_s step 1}
&\intox\mu|\nabla u_s|^2dx\Big|_0^t+\intoxt(\mu+\lambda)|\divv u_s|^2dxd\tau+\intoxt\rho|\dt(u_s)|^2dxd\tau\\
&=-\intoxt\Big(\rho u\cdot\nabla u\Big)\cdot\dt(u_s)dxd\tau-\intoxt\Big(\rho\dt(u_p)\Big)\cdot\dt(u_s)dxd\tau\notag\\
&\qquad+\intoxt\trho^{-1}(\rho-\trho)\nabla\tilde P\cdot\dt(u_s)dxd\tau.\notag
\end{align}
We estimate the right side of \eqref{estimate on u_s step 1} term by term. Using \eqref{L4 estimate on u} and \eqref{estimate on u_p}, the first integral can be bounded by
\begin{align*}
&\Big(\intoxt\rho|u|^2|\nabla u|^2dxd\tau\Big)^\frac{1}{2}\Big(\intoxt\rho|\dt(u_s)|^2dxd\tau\Big)^\frac{1}{2}\\
&\le \tilde C\Big[\int_0^t\Big(\intox\rho|u|^4dx\Big)^\frac{1}{2}\Big(\intox|\nabla u_s|^4dx+\intox|\nabla u_p|^4dx\Big)^\frac{1}{2}d\tau\Big]^\frac{1}{2}\\
&\qquad\qquad\times\Big(\intoxt\rho|\dt(u_s)|^2dxd\tau\Big)^\frac{1}{2}\\
&\le \tilde C\Big[\int_0^t\Big(\intox|\nabla u_s|^2dx\Big)^\frac{1}{4}\Big(\intox|\Delta u_s|^2dx\Big)^\frac{3}{4}d\tau+\int_0^t\Big(\intox|\rho-\tilde\rho|^4dx\Big)^\frac{1}{2}d\tau\Big]^\frac{1}{2}\\
&\qquad\qquad\times\Big(\intoxt\rho|\dt(u_s)|^2dxd\tau\Big)^\frac{1}{2}\\
&\le \tilde C\Big(\intoxt\rho|\dt(u_s)|^2dxd\tau\Big)^\frac{1}{2}\\
&\qquad\qquad\times\Big[\Big(\intoxt|\nabla u_s|^2dxd\tau\Big)^\frac{1}{8}\Big(\intoxt|\Delta u_s|^2dxd\tau\Big)^\frac{3}{8}+1\Big].
\end{align*}
Next to estimate $\dis-\intoxt\Big(\rho\dt(u_p)\Big)\cdot\dt(u_s)dxd\tau$, we differentiate \eqref{def of u_p and u_s}$_1$ with respect to $t$ and use the assumption that $P(\rho)=a\rho$ to obtain
\begin{equation*}
\mu\Delta\dt(u_p)+(\mu+\lambda)\nabla\divv\dt(u_p)=\nabla\divv(-P\cdot u).
\end{equation*}
Using Lemma~\ref{estimates on Lame operator} and the $L^2$-estimate \eqref{L^2 estimate blow up} on $u$, we have
\begin{align}\label{estimate on dt u_p}
\intoxt|\dt(u_p)|^2dxd\tau\le \tilde C\intoxt|P\cdot u|^2dxd\tau\le \tilde C.
\end{align}
Therefore
\begin{align*}
&-\intoxt\Big(\rho\dt(u_p)\Big)\cdot\dt(u_s)dxd\tau\\
&\le\Big(\intoxt\rho|\dt(u_s)|^2dxd\tau\Big)^\frac{1}{2}\Big(\intoxt|\dt(u_p)|^2dxd\tau\Big)^\frac{1}{2}\\
&\le \tilde C\Big(\intoxt\rho|\dt(u_s)|^2dxd\tau\Big)^\frac{1}{2}.
\end{align*}
To estimate $\dis\intoxt\trho^{-1}(\rho-\trho)\nabla\tilde P\cdot\dt(u_s)dxd\tau$, using \eqref{L^2 estimate blow up} and \eqref{lower bound on rho}, we readily have
\begin{align*}
&\intoxt\trho^{-1}(\rho-\trho)\nabla\tilde P\cdot\dt(u_s)dxd\tau\\
&\le \tilde C\Big(\intoxt|\rho-\trho|^2dxd\tau\Big)^\frac{1}{2}\Big(\intoxt\rho|\dt(u_s)|^2dxd\tau\Big)^\frac{1}{2}\\
&\le \tilde C\Big(\intoxt\rho|\dt(u_s)|^2dxd\tau\Big)^\frac{1}{2}.
\end{align*}
Combining the above, we have from \eqref{estimate on u_s step 1} that 
\begin{align}\label{estimate on u_s step 2}
&\intox|\nabla u_s|^2(x,t)dx+\intoxt|\divv(u_s)|^2dxd\tau+\intoxt\rho|\dt(u_s)|^2dxd\tau\notag\\
&\le \tilde C\Big(\intoxt|\nabla u_s|^2dxd\tau\Big)^\frac{1}{4}\Big(\intoxt|\Delta u_s|^2dxd\tau\Big)^\frac{3}{4}+\tilde C.
\end{align}
It remains to estimate the term $\dis\intoxt|\Delta u_s|^2$. Rearranging the terms in \eqref{def of u_p and u_s}$_2$, we have that
\begin{equation*}
\mu\Delta u_s+(\mu+\lambda)\nabla\divv(u_s)=\rho\dt(u_s)+\rho u\cdot\nabla u+\rho\dt(u_p)-\trho^{-1}(\rho-\trho)\nabla\tilde P.
\end{equation*}
Therefore, we can apply Lemma~\ref{estimates on Lame operator} and the bound \eqref{L^2 estimate blow up} to get
\begin{align}\label{estimate on Delta u_s}
&\intoxt|\Delta u_s|^2dxd\tau\\
&\le \tilde C\Big[\intoxt(|\rho\dt(u_s)|^2+|\rho u\cdot\nabla u|^2+|\rho\dt(u_p)|^2+|\rho-\trho|^2)dxd\tau\Big]\notag\\
&\le \tilde C\Big(\intoxt\rho|\dt(u_s)|^2dxd\tau+1\Big).\notag
\end{align}
Applying the estimate \eqref{estimate on Delta u_s} on \eqref{estimate on u_s step 2} and using Gr\"{o}wall's inequality, we conclude that for $0\le t< T^*$,
\begin{equation*}
\intox|\nabla u_s|^2(x,t)dx\le \tilde C,
\end{equation*}
and the result \eqref{estimate on u_s} follows.
\end{proof}
We now give the estimate on $\Phi_2$ as defined in \eqref{def of Phi 2}, which is given in the following lemma:
\begin{lem}\label{bound on Phi 2 lemma}
Assume that \eqref{blow-up 2 not} holds. For any $0\le t< T^*$, 
\begin{equation}\label{bound on Phi 2}
\Phi_2(t)\le \tilde C[\Phi_1(t)+\Phi_3(t)+1].
\end{equation}
\end{lem}
\begin{proof}
Following the steps given in \cite[pp.228-230]{hoff95} (by taking $\sigma\equiv1$), we have
\begin{align}\label{bound on Phi 2 step 1}
&\intox|\dot{u}(x,t)|^2dx+\intoxt|\nabla\dot{u}|^2dxd\tau\notag\\
&\le \tilde C\Big |\sum_{1\le k_i,j_m\le 3}\int_{0}^{t}\int_{\R^3}u^{j_1}_{x_{k_1}}u^{ j_2}_{x_{k_2}}u^{j_3}_{x_{k_3}}dxd\tau\Big|\\
&\qquad\qquad+\tilde C\Big(\intoxt|\nabla u|^4dxd\tau+\Phi_1(t)+1\Big).\notag
\end{align}
The summation term in \eqref{bound on Phi 2 step 1} can be bounded by $\dis\tilde C\intoxt|\nabla u|^3$, and hence it can be bounded by $\tilde C\Phi_3^\frac{1}{2}$. The estimate \eqref{bound on Phi 2} then follows by Cauchy's inequality.
\end{proof}
Finally, we make use of $u_s$ and $u_p$ in \eqref{def of u_p and u_s} to estimate $\Phi_3$:
\begin{lem}\label{bound on Phi 3 lemma}
For any $0\le t< T^*$,
\begin{equation}\label{bound on Phi 3}
\Phi_3(t)\le \tilde C[\Phi_1(t)^\frac{1}{2}+1].
\end{equation}
\end{lem}
\begin{proof}
Using the decomposition \eqref{decomposition on u} on $u$ and the estimates \eqref{estimate on u_p} and \eqref{estimate on u_s}, we have
\begin{align*}
\Phi_3&\le \intoxt|\nabla u_s|^4dxd\tau+\intoxt|\nabla u_p|^4dxd\tau\\
&\le \tilde C\int_0^t\Big(\intox|\nabla u_s|^2dx\Big)^\frac{1}{2}\Big(\intox|\Delta u_s|^2dx\Big)^\frac{3}{2}d\tau+\intoxt|\rho-\tilde\rho|^4dxd\tau\\
&\le \tilde C\Big[\Big(\sup_{0\le \tau\le t}\intox|\Delta u_s(x,\tau)|^2dx\Big)^\frac{1}{2}+1\Big].
\end{align*}
To estimate $\dis\intox|\Delta u_s|^2$, we rearrange the terms in \eqref{def of u_p and u_s}$_2$ to obtain
\[\mu\Delta u_s+(\mu+\lambda)\nabla\divv(u_s)=\rho\dot{u}-\rho\nabla\phi.\]
Therefore Lemma~\ref{estimates on Lame operator} implies that
\begin{align*}
\intox|\Delta u_s|^2dx\le \tilde C\Big[\intox(|\rho\dot{u}|^2+|\rho\nabla\phi|^2)dx\Big]\le \tilde C(\Phi_2+1),
\end{align*}
and the result follows.
\end{proof}
\begin{proof}[Proof of Theorem~\ref{Main thm for gamma=1}]
In view of the bounds \eqref{bound on Phi 1}, \eqref{bound on Phi 2} and \eqref{bound on Phi 3}, one can conclude that for $0\le t< T^*$,
\begin{equation}\label{bound on Phi 1 Phi 2 Phi 3}
\Phi_1(t)+\Phi_2(t)+\Phi_3(t)\le \tilde C.
\end{equation}
Hence using the bound \eqref{bound on Phi 1 Phi 2 Phi 3} and applying the same argument given in the proof of Lemma~\ref{higher estimate on rho and u lem}, for $T\in[0,T^*)$ and $q\in(3,6]$, we also have
\begin{align*}\label{higher estimate on rho and u}
\sup_{0\le t\le T}(\|\rho\|_{H^1\cap W^{1,q}}+\|\nabla u\|_{H^1})\le \tilde C.
\end{align*}
Therefore, similar to the proof of Theorem~\ref{Main thm for gamma>1}, we can extend the strong solution $(\rho,u)$ beyond $t=T^*$, which leads to a contradiction. This completes the proof of Theorem~\ref{Main thm for gamma=1}.
\end{proof}

% ------------------------------------------------------------------------

\subsection*{Acknowledgment}
The author would like to thank the anonymous reviewers for their useful comments which have greatly improved the manuscript. The work described in this paper was partially supported from the Dean's Research Fund of the Faculty of Liberal Arts and Social Science, The Education University of Hong Kong, HKSAR, China (Project No. FLASS/DRF 04634). This work does not have any conflicts of interest.

\bibliographystyle{amsalpha}

\bibliography{References_for_blow_up_large_potential}

% ------------------------------------------------------------------------
\end{document}